\documentclass[11pt,twoside,reqno]{amsart}
\usepackage{amsmath, amsfonts, amsthm, amssymb,amscd, graphicx, amscd}
\usepackage{float,epsf}
\usepackage[english]{babel}
\usepackage{enumerate}
\usepackage{tikz}
\usepackage{hyperref}

\usepackage{srcltx}

\usepackage{geometry}
\usepackage{verbatim}
\usepackage{mathrsfs}

\newtheorem{theorem}{Theorem}[section]
\newtheorem{lemma}[theorem]{Lemma}

\newtheorem{proposition}[theorem]{Proposition}
\newtheorem{corollary}[theorem]{Corollary}
\newtheorem{remark}[theorem]{Remark}

\newtheorem{conjecture}{Conjecture}

\numberwithin{equation}{section}
\numberwithin{figure}{section}

\renewcommand{\div}{\mathrm{div}} 

\def\intave#1{\int_{#1}\hbox{\llap{$\raise2.3pt\hbox{\vrule
height.9pt width7pt}\phantom{\scriptstyle{#1}}\mkern-2mu$}}}

\begin{document}
\title{Some Energy Estimates for Stable Solutions to Fractional Allen-Cahn Equations}

\author{Changfeng Gui}
\address{Department of Mathematics, The University of Texas at San Antonio, San Antonio, TX 78249}
\email{\tt changfeng.gui@utsa.edu}
\author{Qinfeng Li}
 \address{Department of Mathematics,
The University of Texas at San Antonio, San Antonio, TX 78249}
\email{\tt qinfeng.li@utsa.edu}
                  
\date{}

\begin{abstract}
In this paper we study stable solutions to the fractional equation \begin{align}
\label{AC1}
    (-\Delta)^s u =f(u), \quad |u| < 1 \quad  \mbox{in $\mathbb{R}^d$},
\end{align}where $0<s<1$ and $f:[-1,1] \rightarrow \mathbb{R}$ is a $C^{1,\alpha}$ function for $\alpha>\max\{0, 1-2s\}$. We obtain sharp energy estimates for $0<s<1/2$ and rough energy estimates for $1/2 \le s <1$. These lead to a different proof from literature of the fact that when $d=2, \, 0<s<1$, entire stable solutions to \eqref{AC1} are $1$-D solutions.

The scheme used in this paper is inspired by Cinti-Serra-Valdinoci \cite{CSE} which deals with stable nonlocal sets, and Figalli-Serra \cite{FS17} which studies stable solutions to \eqref{AC1} for the case $s=1/2$. 
\end{abstract}

\maketitle

\section{Introduction}
\subsection{Nonlocal Stable De Giorgi Conjeture}
It is well known that for $0<s<1$, the fractional $s$-Laplacian is defined as 
\begin{align}
    (-\Delta)^s u(x):=&C(d,s) (P.V.) \int_{\mathbb{R}^d} \frac{u(x)-u(y)}{|x-y|^{d+2s}}dy\\
    =&\frac{C(d,s)}{2}\int_{\mathbb{R}^d} \frac{2u(x)-u(x+y)-u(x-y)}{|x-y|^{d+2s}}dy,
\end{align}where $C(d,x)$ is a constant such that 
\begin{align}
    \hat{-(\Delta)^s u}(\xi)=|\xi|^{2s}\hat{u}(\xi)
\end{align}

For  $\Omega \subset \mathbb{R}^d$, we consider the fractional Allen-Cahn type equation \begin{align*}
    (-\Delta)^s u =f(u), \quad |u|<1\quad  \, \mbox{in $\Omega$},
\end{align*}
which is the vanishing condition for the first variation of the energy 
\begin{align*}
    \label{totalenergy}
    \mathcal{J}(u,\Omega)=&\mathcal{J}^s(u,\Omega)+\mathcal{J}^P(u,\Omega)\\
    :=&\frac{1}{2}\int\int_{\mathbb{R}^d \times \mathbb{R}^d \setminus (\Omega^c \times \Omega^c)} \frac{|u(x)-u(y)|^2}{|x-y|^{d+2s}}dydx\\
    &+\int_{\Omega} F(u(x))  dx,
\end{align*}up to normalization constants that we omitted for simplicity.

Throughout the paper we assume that $F$ is the primitive function of a given $C^{1,\alpha}$ function $f: [-1,1] \rightarrow \mathbb{R}$, where $\alpha>\max\{0, 1-2s\} $. The regularity of $f$ is to guarantee that any solution $u$ to \eqref{AC1} is in $C^2(\mathbb{R}^d)$ so that the fractional Laplacian is well defined, see for example \cite[Lemma 4.4]{CY14} for the proof. We also throughout the paper assume that $F: \mathbb{R} \rightarrow [0, \infty)$ is a double well potential with two minima $-1$ and $1$. This is the sufficient and necessary condition to guarantee the existence of $1$-D layer solutions to \eqref{AC1}, see \cite[Theorem 2.4]{CY15}. Recall that \textit{layer solutions} are solutions that are monotone in one variable and have limits $\pm 1$ at $\pm \infty$.

In this paper, we study stable solutions to the fractional Allen-Cahn equation \eqref{AC1}. Recall that $u$ is a \textit{stable solution} to \eqref{AC1}, if the second local variation of $\mathcal{J}(\cdot, \mathbb{R}^d)$ at $u$ is nonnegative. Or equivalently,
\begin{align*}
   \int_{\mathbb{R}^d} \left((-\Delta)^s v+f'(u)v\right)v \ge 0, \quad \forall v \in C_0^{2}(\mathbb{R}^d).
\end{align*}
Note that stable solutions include local minimizers or monotone stationary solutions of $\mathcal{J}(\cdot, \mathbb{R}^d)$. Also it is known that $1$-D stable solutions are layer solutions, see the proof of \cite[Lemma 3.1]{DFV18} and \cite[Theorem 2.12]{CY15}.

We would like to study the symmetry results of stable entire solutions to \eqref{AC1}, which is related to the nonlocal version of De Giorgi Conjecture for stable solutions: 
\begin{conjecture}[Nonlocal Stable De Giorgi Conjecture]
\label{stable conjecture}
Let $0<s<1$ and $u$ be a stable solution to \eqref{AC1}, then $u$ is a $1$-D solution for $d \le 7$.
\end{conjecture}

\hfill\\

\subsection{Background and Motivation of Conjecture \ref{stable conjecture}}
In 1979, De Giorgi made the following conjecture on the entire solutions to classical Allen-Cahn equations:
\begin{conjecture}[Classical De Girogi Conjecture]
\label{DeGiorgimonotone}
If $u$ is a solution to the classical Allen-Cahn equation 
\begin{align}
    -\Delta u= u-u^3, \quad  |u|<1 \quad \mbox{in $\mathbb{R}^d$},
\end{align}with $\partial _{x_d}u>0$, then $u$ is a $1$-D solution if $d \le 8$. 
\end{conjecture}

The classical De Giorgi conjecture is closely related to minimal surface theory. If $u$ is a local minimizer to the associated energy funtional \begin{align}
\label{classicalenergyfunctional}
    \mathcal{E}(u, \Omega):=\frac{1}{2}\int_{\Omega} |\nabla u|^2dx+\frac{1}{4}\int_{\Omega}(1-u^2)^2 dx,
\end{align}where $\Omega=\mathbb{R}^d$,  then $u_{\epsilon}(x):=u(x/\epsilon)$ is a minimizer to 
\begin{align*}
    \mathcal{E}_{\epsilon}(v,\epsilon \Omega)=\int_{\epsilon \Omega} \frac{\epsilon}{2}|\nabla v|^2dx +\frac{1}{4\epsilon}\int_{\epsilon \Omega} (1-v^2)^2dx.
\end{align*} Scaling and energy estimates for minimizers imply \begin{align*}
   \mathcal{E}_{\epsilon}(u_{\epsilon}, B_1)=\epsilon^{d-1}\mathcal{E}(u,B_{1/\epsilon}) \le C(d).
   \end{align*}
By Modica-Mortola Gamma convergence result \cite{MM}, $u_{\epsilon} \rightarrow \chi_E-\chi_{E^c}$ in $L^1_{loc}$ for a subsequence $\epsilon_k \rightarrow 0$, and $E$ is a perimeter minimizer in $\mathbb{R}^d$.  If $2\le d \le 8$ and $\partial E$ is a graph, then the classification of entire minimal graphs in $\mathbb{R}^d$ implies that $E$ must be a half space, and thus $\{u_{\epsilon}>t\}$ converge to a half space locally in $L^1$ for $-1<t<1$. Since $\{u_{\epsilon}>t\}={\epsilon}\{u>t\}$, De Giorgi conjectured that $\{u>t\}$ itself has to be a half space for any $t$, even for $u$ to be monotone in direction without being a minimizer.

The case when $d=2$ was proved by Ghoussoub and Gui in \cite{GG}, and the case when $d=3$ was proved by Ambrosio and Cabr\'e in \cite{AC}.   For $d \ge 9$, counterexamples were given by Del Pino, Kowalczyk and Wei \cite{dKW}. The case $4\le d\le 8$ was proved by Savin \cite{Savin} under the additional assumption that 
\begin{align}
\label{limit}
    \lim_{x_d \rightarrow \pm \infty}u(x',x_d)=\pm 1, \quad \mbox{for any $x' \in \mathbb{R}^{d-1}$}
\end{align}
The conjecture remains open for $4\le d \le8$ without the limit condition \eqref{limit}. We remark that in \cite{Savin}, only the minimality of $u$ is used, which is guaranteed by the the monotone condition and \eqref{limit}. We also remark that if the limit in \eqref{limit} is uniform, then Conjecture \ref{DeGiorgimonotone} is true in any dimension $d$ without the monotone assumption. This is proved in \cite{BBG00}, \cite{BHM00} and \cite{Far99} independently.

This conjecture in its full generality remains open.
\hfill\\

In the fractional analogue, if a solution $u$ is a minimizer to the associated energy, then $u_{\epsilon}(x):=u(x/\epsilon)$ is a minimizer to 
\begin{align*}
    \mathcal{J}_{s,\epsilon} (u, \Omega):= 
    \begin{cases}
        \epsilon^{2s-1}\mathcal{J}^s(u,\epsilon\Omega)+\frac{1}{\epsilon} \mathcal{J}^P(u,\epsilon\Omega), \quad & \mbox{if $1/2<s<1$,}\\
        \frac{1}{|\log\epsilon|}\mathcal{J}^s(u,\epsilon\Omega) +\frac{1}{\epsilon|\log\epsilon|}\mathcal{J}^P(u,\epsilon\Omega), \quad &\mbox{if $s=1/2$},\\
        \mathcal{J}^s(u,\epsilon\Omega)+\frac{1}{\epsilon^{2s}}\mathcal{J}^P(u,\epsilon \Omega), \quad & \mbox{if $0<s<1/2$.}
    \end{cases}
\end{align*}    
In \cite{SV12}, Savin and Valdinoci proved that if $\sup_{0<\epsilon<1}J_{s,\epsilon}(u_{\epsilon},\Omega)<\infty$, then $u_{\epsilon} \rightarrow \chi_E-\chi_{E^c}$ in $L^1$ up to a subsequence, where $E$ is a perimeter minimizer in $\Omega$ for $s \in [1/2,1)$ and an $s$-perimeter minimizer in $\Omega$ for $s \in (0,1/2)$.  The classification for global $s$-minimal graphs is the following, which is a combination of several works due to Cafarelli, Figalli, Valdinoci and Savin, see \cite{CV13}, \cite{FV} and \cite{SV13}. 

Let $E$ be an $s$-perimeter graph. Assume that either 
\begin{itemize}
    \item  $d=2,3$,
    \item or $d \le 8$ and $\frac{1}{2}-s \le \epsilon_0$ for some $\epsilon_0>0$ sufficiently small.
\end{itemize}Then $E$ must be a half space. 

It is not known whether the above classification result is optimal, since there are no known examples of $s$-minimal graphs other than hyperplanes, as far as we are aware.

These results motivate the following De Giorgi conjecture in the nonlocal case:
\begin{conjecture}[Nonlocal De Giorgi Conjecture]
\label{monotone conjecture}
Let $0<s<1$ and $u$ be a solution to \eqref{AC1} with
\begin{align}
    \label{monotonecondition}
\partial _{x_d}u>0,
\end{align} then $u$ is a $1$-D solution for $d \le 8$. 
\end{conjecture}

Conjecture \ref{monotone conjecture} has been validated in different cases, according to the following result:
\begin{theorem}
\label{longrange7}
Let $u$ be an entire solution to \eqref{AC1} satisfying \eqref{monotonecondition}, then suppose that either $d=2, 3, \, s\in (0,1)$ or $d=4,\, s=1/2$, then $u$ is $1$-D.
\end{theorem}
Theorem \ref{longrange7} is due to \cite{CSM05} when $d=2,\, s=1/2$, \cite{CY15} and \cite{SV09} when $d=2, \, 0<s<1$, \cite{CC10} when $d=3, \, s=1/2$, \cite{CC14} when $d=3, \, 1/2<s<1$, \cite{DFV18} when $d=3, \, 0<s<1/2$ and \cite{FS17} when $d=4,\, s=1/2$. 

Concerning the nonlocal De Giorgi Conjecture in higher dimensions with the additional limit condition \eqref{limit} or with minimality condition, the best known results are the following two theorems, which were proved in \cite{SAV16} when $s\in (1/2,1)$, \cite{SAV18} when $s=1/2$ and \cite{DSV16} when $s \in (1/2-\epsilon_0, 1]$. 

\begin{theorem}
\label{opennonlocalbothcondition}
Let $d \le 8$. Then, there exists $\epsilon_0 \in (0,1/2]$ such that for any $s \in (1/2-\epsilon_0, 1]$, the following statement holds true:

Let $u$ be an entire solution to \eqref{AC1} satisfying \eqref{limit} and \eqref{monotonecondition} , then $u$ is $1$-D. 
\end{theorem}

\begin{theorem}
\label{minimizer}
Let $d \le 7$. Then, there exists $\epsilon_0 \in (0,1/2]$ such that for any $s \in (1/2-\epsilon_0, 1]$, the following statement holds true:

Let $u$ be an entire solution to \eqref{AC1} which is a minimizer of $\mathcal{J}(\cdot, \mathbb{R}^d)$, then $u$ is $1$-D. 
\end{theorem}
A counterexample for $d=9, \, 1/2<s<1$ is announced by H. Chan, J. D´avila, M. del Pino, Y. Liu and J. Wei, see the comments after \cite[Theorem 1.3]{CLW17}. The other cases remain open.

\hfill\\

Motivated by Conjecture \ref{monotone conjecture}, it is natural to study the stable De Giorgi Conjecture, that is, Conjecture \ref{stable conjecture}. This is because, on the one hand, it is well known that monotone solutions to \eqref{AC1} are stable solutions. On the other hand, a further relation between stable solutions and monotone solutions to \eqref{AC1} is given in the following remark:
\begin{remark}
\label{connection}
If any entire stable solution to \eqref{AC1} in $\mathbb{R}^{d-1}$ is $1$-D, then any monotone solution to \eqref{AC1}  in $\mathbb{R}^d$ is also $1$-D for $d \le 3, s\in (0,1)$ and for $4 \le d \le 7, \, s \in (1/2-\epsilon_0,1)$, where $\epsilon_0 \in (0,1/2]$ is some constant.
\end{remark}


Remark \ref{connection} is well known by experts, but we haven't seen a proof in the literature. We will prove Remark \ref{connection} in Appendix.


Because of the connection between monotone solutions and stable solutions as revealed in Remark \ref{connection}, it is important to study Conjecture \ref{stable conjecture}.

\subsection{Previous Results on Conjecture \ref{stable conjecture}}

For $d=2$, Conjecture \ref{stable conjecture} was validated by Cabr\'e and  Sol\'a-Morales in \cite{CSM05} for $s=1/2$, and by Cabr\'e and Sire in \cite{CY15} and by Sire and Valdinoci in \cite{SV09} for every fractional power $0<s<1$ with different proofs, all of which require Cafarelli-Silvestre extension \cite{CS07} and the stability of $s$-harmonic extension $U$ in $\mathbb{R}^d \times (0,\infty)$.
The stability condition used in these references is the following: 

\begin{remark}
In \cite{CSM05}, \cite{SV09} and \cite{CY15}, the stability of solution $u$ to \eqref{AC1} was understood in the sense that the second local variation of the extension energy
\begin{align*}
    \mathcal{E}(U;\mathbb{R}^{d+1}_+)=\frac{1}{2} \int_0^{\infty} \int_{\mathbb{R}^d} z^{1-2s}|\nabla U|^2 dx dt+\int_{\mathbb{R}^d}F(u(x)) dx
\end{align*} is nonnegative at $U$, where $U$ is the Cafarelli-Silvestre extension of $u$ which solves
\begin{align*}
\begin{cases}
    (i) \,\div (t^{1-2s} U(x,t))=0 \quad& \mbox{in $\mathbb{R}^d \times (0,\infty)$}\\
   (ii) \, c_s\lim_{t \rightarrow 0}t^{1-2s}\partial_t U(x,t)=f(U(x,0)) \quad & \mbox{on $\partial \mathbb{R}^{d+1}_+$}
    \end{cases}
\end{align*} with boundary condition $U(x,0)=u(x)$, where $c_s$ is a constant which is discussed in 
\cite[Remark 3.11]{CY14}. It appears that this stable assumption is stronger than ours which just considers local variations on $\mathbb{R}^d$ instead of $\mathbb{R}^d \times (0,\infty)$. Later it was shown in \cite[Proposition 2.3]{DFV18} that the two stable definitions are equivalent for every fractional power $0<s<1$.
\end{remark}

For $3 \le d \le 8$ and $0<s<1$, Conjecture \ref{stable conjecture} remains open except for the case $d=3$ and $s=1/2$. In fact, it has been recently validated by Figalli and Serra in \cite{FS17} without using extension results in \cite{CS07}. Figalli and Serra utilized the local BV estimates scheme originally developed by Cinti, Serra and Valdinoci in \cite{CSE} for stable sets (see Definition 1.6 there), together with the following sharp interpolation inequality \begin{align}
\label{fsstep2}
    \mathcal{J}^{1/2}(u,B_1) \le C(d)\log L_0\left(1+\int_{B_2} |\nabla u| dx\right),
\end{align}where $L_0\ge 2$ is an upper bound for $\Vert \nabla u\Vert _{\infty}$, to prove the following energy estimates in any dimension $d$ and $s=1/2$, which is a key ingredient to validate Conjecture \ref{stable conjecture}.
\begin{proposition}(\cite[Proposition 1.7]{FS17})
\label{1.7}
If $u$ is a stable solution to \eqref{AC1}, then \begin{align}
\label{fse1}
    \int_{B_R}|\nabla u| \le CR^{d-1}\log (M_0R)
\end{align}
and
\begin{align}
\label{fse2}
    \mathcal{J}^{1/2}(u,B_R) \le CR^{d-1} \log^2(M_0R)
\end{align}
where $C$ is a universal constant depending only on $d$ and $\alpha$, and $M_0 \ge 2$ is an upper bound for the H\"older norm of $f$.
\end{proposition}With \eqref{fse1} and \eqref{fse2} being applied in the local BV estimate scheme, and by a bootstrap argument, Figalli and Serra were able to prove Conjecture \ref{stable conjecture} for $d=3$ and $s=1/2$.

\hfill\\

\subsection{Our Contribution in this Paper}
Proving energy estimates like \eqref{fse1} and \eqref{fse2} for stable solutions to \eqref{AC1} for every fractional power $s \in (0,1)$ is definitely a decisive step to solve Conjecture \ref{stable conjecture}. 

We have observed that actually suitable adaptation of the local BV estimate scheme used in \cite{FS17} together with a generalized form of \eqref{fsstep2} can produce energy estimates for stable solutions in arbitrary dimension $d$ and energy $0<s<1$. We prove: \begin{proposition}
\label{generalenergyestimates}
Let $u \in C^2(\mathbb{R}^d)$ be a stable solution to
\begin{align}
\label{AC}
    (-\Delta)^s u =f(u), \quad |u| \le 1 \quad  \mbox{in $\mathbb{R}^d$},
\end{align}then there exists constant $C_1=C_1(d,s)$ and $C_2=C_2(d,s,f)$ such that for any ball $B_R \subset \mathbb{R}^d, \, R\ge 1$, we have
\begin{align}
    \label{gle1}
\int_{B_R}|\nabla u| \le \begin{cases}
C_1 R^{d-1} \quad & 0<s<\frac{1}{2}\\
C_2 R^{d+2s-2}\log(M_0R) \quad &\frac{1}{2}\le s<1
\end{cases}
\end{align}
and 
\begin{align}
    \label{gle2}
 \mathcal{J}^{s}(u,B_R) \le \begin{cases}
 C_1 R^{d-2s} \quad &0<s<\frac{1}{2}\\
 C_2 R^{d+2s-2}\log^2(M_0R) \quad &\frac{1}{2}\le s<1,
 \end{cases}
\end{align}where $M_0 \ge 2$ is an upper bound for $L^{\infty}$ norm of $f$. 
\end{proposition}
Note that it is easy to see that for a bounded Lipschitz function $u$, the natural growth for fractional energy is
\begin{align*}
    \mathcal{J}^s(u,B_R) \le CR^d,
\end{align*}see for example Lemma \ref{embedding} below. Such estimate is too rough. It is with the stability condition of $u$ that we can derive a sharper fractional energy growth estimate \eqref{gle2} than the natural one.

\eqref{gle1} and \eqref{gle2} are sharp for the case $0<s<1/2$, in the sense that the local minimizers do satisfy same estimates, which are optimal, see \cite{SV14} and \cite{PSV}. Although for the case $1/2\le s<1$, our energy estimates are not optimal, the adaptation of local BV estimates scheme in \cite{CSE} and \cite{FS17} together with our energy estimates can also give a different proof to validate Conjecture \ref{stable conjecture} for the case $d=2, \, 0<s<1$, see Theorem \ref{stable2d}.

We remark that when $s=1/2$, $C_2$ does not depend on $f$ by keeping track of the constant in our proof. Thus in this case, the second inequalities in \eqref{gle1} and \eqref{gle2} coincide with \eqref{fse1} and \eqref{fse2} in Proposition \ref{1.7}. 

We also remark that the key of proving \eqref{fsstep2} is by \cite[Lemma 2.1]{FJ14} (or \cite[Theorem 2.4]{FS17}), whose proof was based on by Plancherel formula plus some delicate estimates. The proof seems to work only for the case $s=1/2$. We give a different proof in this paper that actually works for all cases $1/2 \le s<1$. In fact, we can prove the following result, which might have independent interest. 
\begin{lemma}
\label{embedding}
For any ball $B_R \subset \mathbb{R}^d$ and $u$ which belongs to appropriate space with $|u|\le 1$, and let $s \in (0,1/2)$, there exists universal constant $C=C(d,s)>0$ such that for any $R \ge 1$,
\begin{align}
\label{lessthan1/2}
    \mathcal{J}^s(u,B_{R}) \le C\left( \int_{B_{2R}}|\nabla u|dx+R^{d-2s}+R^d\right).
\end{align}
If $1/2 \le s <1$ and $u$ is assumed to be a Lipschitz function with $\Vert \nabla u \Vert _{L^{\infty}(B_R)} \le L_0, \, L_0 \ge 2$, then there exists $C=C(d,s)>0$ such that
\begin{align}
\label{biggerthan1/2}
    \mathcal{J}^s(u,B_{R}) \le C\left( R^{d-2s}+L_0^{2s-1}\log (2L_0R)\int_{B_{2R}} |\nabla u| \right).
\end{align}
\end{lemma}
Note that when $R=1$ and $s=1/2$, \eqref{biggerthan1/2} is exactly \eqref{fsstep2}.

It is with Lemma \ref{embedding} and the adaptation of local BV estimate for arbitary fractional powers $s \in (0,1)$, we can prove Proposition \ref{generalenergyestimates}.

\hfill\\

\begin{remark}
Only after this work was completed, we have noticed that Cinti has mentioned in her survey article \cite{Cinti} that she, Cabr\'e and Serra are carrying out a careful study on nonlocal stable phase transitions in \cite{CCJ}, which has not been posted yet. As Cinti mentioned, they will state energy estimates, density estimates, convergence of blow-down and some new classification results for stable solutions for fractional powers $0<s<1/2$. While our focus in this paper is to exploit the ideas in \cite{CSE} and \cite{FS17} to prove energy estimates for all fractional powers $0<s<1$, as best as we can do at this moment.
\end{remark}

\subsection{Outline of this paper}

In section 2 we prove Lemma \ref{embedding}. In section 3, we validate the BV estimate scheme for any fractonal power $s \in (0,1)$ and use it to prove Proposition \ref{generalenergyestimates}, and then as an application we validate Conjecture \ref{stable conjecture} for the case $d=2, \, s\in (0,1)$. In the appendix we prove Remark \ref{connection}.

\section{Proof of Lemma \ref{embedding}}
In this section we prove Lemma \ref{embedding}. We first recall the fractional Sobolev embedding theorem:
\begin{proposition}(see \cite[Proposition 2.2]{NPV})
\label{prop2.2}
For $s \in (0,1)$, $p \ge 1$ and $B_R \subset \mathbb{R}^d$, we have
\begin{align}
    \Vert u \Vert_{W^{s,p}(B_R)} \le C(d,p,s) \Vert u \Vert_{W^{1,p}({B_R})}
\end{align}
\end{proposition}

In order to prove Lemma \ref{embedding}, we also need to prove:
\begin{lemma}
\label{betterestimate}
Assume $|u|\le 1$ and $\Vert \nabla u \Vert _{L^{\infty}(B_1)} \le L_0$, where $L_0 \ge 2$, then for $s \in [1/2,1)$, 
\begin{align}
    \int_{B_1}\int_{B_1}\frac{|u(x)-u(y)|^2}{|x-y|^{d+2s}}dx dy \le  &\frac{1}{1-s}d\omega_dL_0^{2s-1}\Big((2-2s)\log(2L_0)+1\Big)\int_{B_1}|\nabla u(x)| dx\\
    \label{6.5}
    =& C(d,s)L_0^{2s-1}\log(L_0)\int_{B_1}|\nabla u(x)| dx
\end{align}where $\omega_d$ is the volume of the unit ball in $\mathbb{R}^d$.
\end{lemma}
\begin{proof}
We estimate
\begin{align*}
     &\int_{B_1}\int_{B_1}\frac{|u(x)-u(y)|^2}{|x-y|^{d+2s}}dy dx\\
     = & \int_{B_1}\int_{B_2}\chi_{\{z:\, x+z \in B_1\}}|u(x+z)-u(x)|^{2-2s}\frac{|u(x+z)-u(x)|^{2s}}{|z|^{d+2s}}dz dx \\
    =& \int_{B_1}\int_0^2\int_{B_2}\chi_{\{z: \,|u(x+z)-u(x)|^{2-2s}>t, \, x+z \in B_1\}}\frac{|u(x+z)-u(x)|^{2s}}{|z|^{d+2s}}dzdt dx \\
    \le & \int_{B_1}\int_0^2\int_{B_2}\chi_{\{z: |z| > \frac{t^{\frac{1}{2-2s}}}{M_0}, \, x+z \in B_1\}}\frac{|u(x+z)-u(x)|^{2s}}{|z|^{d+2s}}dzdt dx \\
    = & \int_{B_1}\int_0^2\int_{B_2}\chi_{\{z: |z| > \frac{t^{\frac{1}{2-2s}}}{M_0}, \, x+z \in B_1\}}\frac{|u(x+z)-u(x)|}{|z|^{d+1}}\frac{|u(x+z)-u(x)|^{2s-1}}{|z|^{2s-1}}dzdt dx \\
    \le & M_0^{2s-1}\int_{B_1}\int_0^2\int_{B_2}\chi_{\{z: |z| > \frac{t^{\frac{1}{2-2s}}}{M_0}, \, x+z \in B_1\}} \frac{\int_0^1 |\nabla u(x+rz)| dr}{|z|^{d}}dzdt dx\\
     = & M_0^{2s-1}\int_0^1\int_0^2\int \chi_{\{z\in B_2: |z| > \frac{t^{\frac{1}{2-2s}}}{M_0}\}}\int_{B_1}\chi_{\{x \in B_1: \, x+z\in B_1\}}\frac{|\nabla u(x+rz)|}{|z|^{d}}dxdz dtdr\\
     \le & M_0^{2s-1 }\int_{B_1}|\nabla u(x)|dx\int_0^1\int_0^2\int \chi_{\{z\in B_2: |z| > \frac{t^{\frac{1}{2-2s}}}{M_0}\}}\frac{1}{|z|^d}dz dtdr\\
     = & M_0^{2s-1 }\int_{B_1}|\nabla u(x)|dx\int_0^1\int_0^{(2M_0)^{2-2s} \wedge 2}\int \chi_{\{z\in B_2: |z| > \frac{t^{\frac{1}{2-2s}}}{M_0}\}}\frac{1}{|z|^d}dz dtdr\\
     =& dw_{d}M_0^{2s-1 }\int_{B_1}|\nabla u(x)|dx\int_0^1\int_0^{(2M_0)^{2-2s} \wedge 2}\left(\log(2M_0)-\frac{1}{2-2s}\log t\right) dtdr\\
     = & 2\wedge (2M_0)^{2-2s}d\omega_dM_0^{2s-1 }\int_{B_1}|\nabla u(x)|dx\left(\log (2M_0)+\frac{1-log\left(2\wedge (2M_0)^{2-2s}\right)}{2-2s}\right)\\
     \le& \frac{1}{1-s}d\omega_dM_0^{2s-1}\Big((2-2s)\log(2M_0)+1\Big)\int_{B_1}|\nabla u(x)| dx,
\end{align*}where in the above we have used that the layer-cake formula for nonnegative function $g \in L^1(d\lambda)$, $\lambda$ being a Radon measure, 
\begin{align*}
    \int g(x) H(x) d\lambda =\int_{0}^{\Vert g \Vert _{\infty}}\int_{\{x :g(x)>t\}}H(x)d\lambda dt,
\end{align*} that for $s \in [1/2,1)$,
\begin{align*}
    \{z: |u(x+z)-u(x)|^{2-2s}>t\} \subset \{z: |z| > \frac{t^{\frac{1}{2-2s}}}{M_0}\},
\end{align*}
and that $x\in B_1, \, x+z \in B_1$ implies
\begin{align*}
    x+rz=r(x+z)+(1-r)x \in B_1, \quad  \mbox{by convexity of $B_1$}
\end{align*}
\end{proof}


The following corollary can be obtained by modifying the proof of Lemma \ref{betterestimate}, and it might have some independent interest.
\begin{corollary}
Let $L_0 \ge 2$. then for any $|u| \le 1$, $\Vert \nabla u \Vert _{L^{\infty}(B_1)} \le L_0$ and any $p>1$, the following estimate holds:
\begin{align*}
    \int_{B_1}\int_{B_1}\frac{|u(x)-u(y)|^p}{|x-y|^{d+1}} dx dy\le C(d,p) \log(L_0) \int_{B_1} |\nabla u(x)| dx.
\end{align*}
\end{corollary}
We omit the proof of this corollary.
\hfill\\

Now we prove Lemma \ref{embedding}.
\begin{proof}[Proof of Lemma \ref{embedding}]
For $0<s<1/2$, we estimate
\begin{align*}
    \mathcal{J}^s(u,B_{R})=& \int\int_{\mathbb{R}^d \times \mathbb{R}^d \setminus (B_{R}^c \times B_{R}^c)} \frac{|u(x)-u(y)|^2}{|x-y|^{d+2s}}dydx\\
    \le & \int\int_{B_{2R} \times B_{2R}} \frac{|u(x)-u(y)|^2}{|x-y|^{d+2s}}dydx +2\int\int_{B_{R} \times B_{2R}^c} \frac{|u(x)-u(y)|^2}{|x-y|^{d+2s}}dydx\\
    \le & 2\int\int_{B_{2R} \times B_{2R}} \frac{|u(x)-u(y)|}{|x-y|^{d+2s}}dydx+C(d,s)R^{d-2s}, \quad \mbox{since $|u|\le 1$}\\
    =& 2[u]_{W^{2s,1}(B_{2R})}+C(d,s)R^{d-2s}\\
    \le & C(d,s)\Vert u \Vert_{W^{1,1}(B_{2R})}+C(d,s)R^{d-2s}, \quad \mbox{by Proposition \ref{prop2.2}}\\
    \le & C(d,s)\left( \int_{B_{2R}}|\nabla u|dx+R^{d-2s}+R^d\right).
\end{align*}
This concludes \eqref{lessthan1/2}.

Let us now prove the lemma for the case $1/2 \le s<1$. For any ball $B_R(x_0) \subset \mathbb{R}^d$, we let $u_R:=u(x_0+Rx)$, and thus $\Vert \nabla u_R \Vert_{L^{\infty}} =R \Vert \nabla u \Vert_{L^{\infty}}\le   RL_0$. 

By applying Lemma \ref{betterestimate} to $u_R$ and using the scaling properties
\begin{align*}
    \mathcal{J}^s(u_R, B_1) =R^{2s-d} \mathcal{J}(u,B_R(x_0)) \quad \mbox{and } \quad \int_{B_1} |\nabla u_R| dx =R^{1-d} \int_{B_R(x_0)} |\nabla u| dx,
\end{align*}
we thus derive
\begin{align}
\label{byscaling}
    \int\int_{B_{R} \times B_{R}} \frac{|u(x)-u(y)|^2}{|x-y|^{d+2s}}dydx \le C(d,s)L_0^{2s-1}\log(RL_0)\int_{B_{R}}|\nabla u(x)| dx.
\end{align}
Therefore, \eqref{biggerthan1/2} is from the following straightforward computation
\begin{align*}
    \mathcal{J}^s(u,B_{R})=& \int\int_{\mathbb{R}^d \times \mathbb{R}^d \setminus (B_{R}^c \times B_{R}^c)} \frac{|u(x)-u(y)|^2}{|x-y|^{d+2s}}dydx\\
    \le & \int\int_{B_{2R} \times B_{2R}} \frac{|u(x)-u(y)|^2}{|x-y|^{d+2s}}dydx +2\int\int_{B_{R} \times B_{2R}^c} \frac{|u(x)-u(y)|^2}{|x-y|^{d+2s}}dydx\\
    \le & \int\int_{B_{2R} \times B_{2R}} \frac{|u(x)-u(y)|^2}{|x-y|^{d+2s}}dydx+C(d,s)R^{d-2s}, \quad \mbox{since $|u|\le 1$}\\
    \le &  C(d,s)\left( R^{d-2s}+L_0^{2s-1}\log (2L_0R)\int_{B_{2R}} |\nabla u(x)|dx \right), \quad \mbox{by \eqref{byscaling}.}
\end{align*}
\end{proof}

\section{Local BV estimate scheme for any power $0<s<1$ and Proof of Proposition \ref{generalenergyestimates}}
As we mentioned in introduction, the local BV estimate scheme was first developed in \cite{CSE} and adapted by Figalli and Serra in \cite{FS17} for the study of stable solutions to \eqref{AC1} when $s=1/2$. In this section we show that thanks to Lemma \ref{embedding}, the scheme can be applied to give certain energy estimates for every fractional power $0<s<1$, as stated in Proposition \ref{generalenergyestimates}.

First, to utilize the stability condition of solution $u$ to \eqref{AC1}, following \cite{FS17}, see also \cite[Lemma 4.3]{BV16}, we construct suitable variations of energy with respect to a direction $\bf{v}$, where $\bf{v}$ is a fixed unit vector in $\mathbb{R}^d$.

Let $R \ge 1$ and \begin{align*}
    \psi_{t,\bf{v}}(x):=x+t \phi(x) \bf{v},
\end{align*}
where \begin{align}
\label{linearcutoff}
    \phi(x)=\begin{cases} 1,\quad &|x|\le \frac{R}{2}\\
    2-2\frac{|x|}{R},\quad &\frac{R}{2}\le |x| \le R\\
    0, \quad &|x| \ge R.
    \end{cases}
\end{align}
It is clear that when $|t|$ small, $\psi_{t,\bf{v}}$ is a Lipschitz diffeomorphism, and thus it has an inverse. Define
\begin{align*}
    P_{t, \bf{v}}u(x):=u\left(\psi_{t,\bf{v}}^{-1}(x)\right).
\end{align*}
\begin{remark}
\label{variation}
It is clear that for $x \in B_{1/2}$, if $|t|$ is small, then $P_{t, \bf{v}}u(x)=u(x-t\bf{v})$.
\end{remark}

To simplify notation, we define the second variation operator $\Delta ^t_{\bf{v}\bf{v}}$ with respect to $\bf{v}$ on any functional $\mathcal{J}$ to be as 
\begin{align*}
    \Delta ^t_{\bf{v}\bf{v}}\mathcal{J}(u,\Omega):=\mathcal{J} (P_{t,\bf{v}}u,\Omega)+ \mathcal{J}(P_{-t,\bf{v}}u,\Omega)-2\mathcal{J}(u,\Omega).
\end{align*}

The following estimate for the second variation of fractional energy is proved in \cite[Lemma 4.3]{BV16} and \cite[Lemma 2.1]{FS17}. For the courtesy of reader, we include a proof.
\begin{lemma}
\label{domainvariation}
\begin{align*}
    \Delta^t_{\bf{v},\bf{v}}\mathcal{J}^s(u,B_R) \le C(d,s)t^2\frac{\mathcal{J}^s(u,B_R)}{R^2}, \quad \forall R\ge 1.
\end{align*}
\end{lemma}
\begin{proof}
We start with more general domain variations as follows. We consider the map
\begin{align}
    F_t(x):=x+t \eta (x),
\end{align}where $\eta$ is a smooth vector field vanishing outside $B_R$. We set
\begin{align}
    P_tu(x):=u(F_t^{-1}(x)).
\end{align}
We estimate
\begin{align*}
    \Delta^t\mathcal{J}^s(u,B_R):=\mathcal{J}^s(P_tu,B_R)+\mathcal{J}^s(P_{-t}u,B_R)-2\mathcal{J}^s(u,B_R)
\end{align*}
We use $\tilde{B_R}$ to denote $\mathbb{R}^d \times \mathbb{R}^d \setminus (B_R \times B_R)$. In the following computation, $z=x-y$ and $\epsilon(x,y):=\frac{\eta(x)-\eta(y)}{|x-y|}$. Since the Taylor expansion of the Jacobian of $F_t$ is 
\begin{align*}
    JF_t=1+t\div \eta+t^2 A(\eta)+O(t^3),
\end{align*}where
\begin{align*}
    A(\eta)=\frac{(\div \eta)^2 -tr(\nabla \eta)^2}{2},
\end{align*}
we can compute
\begin{align*}
    \Delta^t\mathcal{J}^s(u,B_R)=&\int\int_{\tilde{B_R}}|u(x)-u(y)|^2\Big(K(z+t\epsilon|z|)(1+t\div \eta(x)+A(\eta(x))t^2)(1+t\div \eta(y)+A(\eta(y))t^2)\\
    &+K(z-t\epsilon|z|)(1-t\div \eta(x)+A(\eta(x))t^2)(1-t\div \eta(y)+A(\eta(y))t^2)-2K(z)\Big)dydx\\
    :=& \int\int_{\tilde{B_R}}|u(x)-u(y)|^2 e(x,y,\eta, R) dydx,
\end{align*}where $K(z)=\frac{1}{|z|^{d+2s}}$.
Use that 
\begin{align*}
    K(az)=|a|^{-d-2s}K(z), \quad \forall a \in \mathbb{R},
\end{align*}we have
\begin{align*}
    &e(x,y,\eta, R)\\
    =&K(z)\Big(K(\frac{z}{|z|}+t\epsilon)(1+t\div \eta(x)+A(\eta(x))t^2)(1+t\div \eta(y)+A(\eta(y))t^2)\\
    &+K(\frac{z}{|z|}-t\epsilon)(1-t\div \eta(x)+A(\eta(x))t^2)(1-t\div \eta(y)+A(\eta(y))t^2)-2K(\frac{z}{|z|})\Big)\\
    =&K(z)\Big(\left(K(z/|z|)+t\nabla K(z/|z|)\epsilon+\frac{t^2}{2} <\nabla^2 K(z/|z|)\epsilon, \epsilon>+O(t^3)\right)(1+t\div \eta(x)+A(\eta(x))t^2)\\
    &\cdot(1+t\div \eta(y)+A(\eta(y))t^2)+ \left(K(z/|z|)-t\nabla K(z/|z|)\epsilon+\frac{t^2}{2} <\nabla^2 K(z/|z|)\epsilon, \epsilon>+O(t^3)\right)\\
    &\cdot(1-t\div \eta(x)+A(\eta(x))t^2)(1-t\div \eta(y)+A(\eta(y))t^2)-2K(z/|z|)\Big)\\
    =&2K(z)t^2\Big(A(\eta(x))+A(\eta(y))+\div \eta(x)\div \eta(y)+(\div \eta(x)+\div \eta (y))\nabla K(z/|z|) \epsilon+ <\nabla^2 K(z/|z|)\epsilon, \epsilon>\Big)\\
    &+O(t^3)\\
    \le & C(d,s) \Vert \nabla \eta \Vert_{L^{\infty}(B_R)}^2K(z)t^2
\end{align*}
In particular, if we choose $\eta(x)=\phi(x) \bf{v}$, where $\phi$ is given as \eqref{linearcutoff} and $\mbox{\bf{v}} \in S^{d-1}$, then we have
\begin{align*}
    \Delta^t_{\bf{v},\bf{v}}\mathcal{J}^s(u,B_R) \le C(d,s)t^2\frac{\mathcal{J}^s(u,B_R)}{R^2}.
\end{align*}
\end{proof}

Next, we prove the following identity related to nonlocal fractional energy, which was implicitly used in the proof of \cite[Lemma 2.2]{FS17}. 
\begin{lemma}
\label{identity}
Let $\Omega \subset \mathbb{R}^d$. For any functions $u,v$ in appropriate spaces, let $u \lor v:=\max\{u,v\}$ and $u \land v:=\min\{u,v\}$, we have the identity
\begin{align}
    \mathcal{J}^s(v,\Omega)+\mathcal{J}^s(u,\Omega)-\mathcal{J}^s(u \lor v,\Omega)-\mathcal{J}^s(u \land v,\Omega) =2 \int\int_{\mathbb{R}^d \times \mathbb{R}^d \setminus (\Omega^c \times \Omega^c)}(v-u)_+(x)(v-u)_-(y)K(x-y)dydx,
\end{align}where $K(z)=\frac{1}{|z|^{d+2s}}$, $(v-u)_+=(v-u)\lor 0$ and $(v-u)_-=(v-u) \land 0$.
\end{lemma}
\begin{proof}
Define sets
\begin{align*}
    A:=\{x \in \mathbb{R}^d: v(x)>u(x)\}
\end{align*} and
\begin{align*}
    \tilde{\Omega}:=\mathbb{R}^d \times \mathbb{R}^d \setminus (\Omega^c \times \Omega^c).
\end{align*}
Then we calculate
\begin{align*}
    &\mathcal{J}^s(v,\Omega)+\mathcal{J}^s(u,\Omega)-\mathcal{J}^s(u \lor v,\Omega)-\mathcal{J}^s(u \land v,\Omega)\\
     =&
    \int \int_{(A \times A^c) \cap \tilde{\Omega}}\left(|v(x)-v(y)|^2-|v(x)-u(y)|^2\right)K(x-y)dydx\\
    &+\int \int_{(A^c \times A) \cap \tilde{\Omega}}\left(|v(x)-v(y)|^2-|u(x)-v(y)|^2\right)K(x-y)dydx\\
    &+\int \int_{(A \times A^c) \cap \tilde{\Omega}}\left(|u(x)-u(y)|^2-|u(x)-v(y)|^2\right)K(x-y)dydx\\
    &+\int \int_{(A^c \times A) \cap \tilde{\Omega}}\left(|u(x)-u(y)|^2-|v(x)-u(y)|^2\right)K(x-y)dydx\\
    =&2\int \int_{(A \times A^c) \cap \tilde{\Omega}}\left((v(x)-u(x))(u(y)-v(y))\right)K(x-y)dydx\\
    &+2\int \int_{(A^c \times A) \cap \tilde{\Omega}}\left((u(x)-v(x))(v(y)-u(y))\right)K(x-y)dydx\\
    =&2 \int\int_{\mathbb{R}^d \times \mathbb{R}^d \setminus (\Omega^c \times \Omega^c)}(v-u)_+(x)(v-u)_-(y)K(x-y)dydx,
\end{align*}
\end{proof}

\begin{remark}
The above lemma implies
\begin{align*}
    \mathcal{J}^s(v,\Omega)+\mathcal{J}^s(u,\Omega) \ge \mathcal{J}^s(u \lor v,\Omega)+\mathcal{J}^s(u \land v,\Omega),
\end{align*}and $"="$ holds only if either $v \le u$ or $v \ge u$ in $\Omega$.  To our knowledge this was first used in \cite[Corollary 3]{PSV}, and it is really reveals the nonlocal feature of fractional energies.
\end{remark}

By using the matrix determinant lemma
\begin{align*}
    det(I+\alpha \otimes \beta)=1+\alpha \cdot \beta
\end{align*}where $\alpha, \beta$ are two vectors, one can also check that 
\begin{align}
\label{ganji2}
    \Delta^t_{\bf{v}\bf{v}}\mathcal{J}^P(u,B_1)=0.
\end{align}
This together with lemma \ref{domainvariation} immediately yields

\begin{lemma}
\label{2.1}
There exists universal constant $C=C(d,s)>0$ such that
\begin{align*}
    \Delta ^t_{\bf{v}\bf{v}}\mathcal{J}(u,B_R) \le Ct^2\mathcal{J}^s(u,B_R)/R^2.
\end{align*}
\end{lemma}

\hfill\\

For the rest, unless otherwise specified, we write $C$ as various universal constants depending on $d$ and $s$.

The next lemma, which is from \cite[Lemma 2.2]{FS17}, dealing with the case $s=1/2$ and in the same spirit of \cite[Lemma 2.5]{CSE},  gives upper bound for the interior BV-norm of $u$ by the $s$-fractional energy in a larger ball. Again, the proof in \cite[Lemma 2.2]{FS17} works for all fractional powers $0<s<1$. We state the result and include the proof as courtesy to the readers.
\begin{lemma}
\label{2.2}
Let $u$ be a stable solution to \eqref{AC}, then there exists a universal constant $C=C(d,s)$ such that for any $R \ge 1$,
\begin{align}
\label{1}
    \left(\int_{B_{1/2}} (\partial_{\bf{v}}u(x))_+dx\right) \left( \int_{B_{1/2}} (\partial_{\bf{v}}u(y))_-dy \right) \le C\mathcal{J}^s(u,B_R)/R^2
\end{align}and
\begin{align}
\label{2}
    \int_{B_{1/2}}|\nabla u(x)| dx \le C\left(1+\sqrt{\mathcal{J}^s(u,B_1)}\right).
\end{align}
\end{lemma}
\begin{proof}
Let $\bar{u}=\max\{P_{t, \bf{v}}u, u\}$ and $\underline{u}=\min\{P_{t, \bf{v}}u, u\}$. By Lemma \ref{identity} and Remark \ref{variation}, we have
\begin{align*}
    &\mathcal{J}^s(\bar{u},B_R)+\mathcal{J}^s(\underline{u},B_R)+ 2\int_{B_{1/2}}\int_{B_{1/2}}\frac{\left(\mbox{$u(x-t\bf{v})$}-u(x)\right)_+\left(\mbox{$u(y-t\bf{v}$})-u(y)\right)_-}{|x-y|^{d+2s}} dydx\\
    &\le \mathcal{J}^s(P_{t, \bf{v}}u,B_R)+\mathcal{J}^s(u,B_R).
\end{align*}
We also have
\begin{align*}
    &\mathcal{J}^P(\bar{u},B_R)+\mathcal{J}^P(\underline{u},B_R)\\
    =& \int_{\{P_{t, \bf{v}}u>u\}\cap B_R} F(P_{t, \bf{v}}u) +\int_{\{P_{t, \bf{v}}u<u\}\cap B_R} F(u)+\int_{\{P_{t, \bf{v}}u<u\}\cap B_R} F(P_{t, \bf{v}}u) +\int_{\{P_{t, \bf{v}}u>u\}\cap B_R} F(u)\\
    =& \mathcal{J}^P(P_{t,\bf{v}}u,B_R)+\mathcal{J}^P(u,B_R).
\end{align*}
Since $|x-y|<1$ when $x, y \in B_{1/2}$, we have
\begin{align}
\label{6}
    \mathcal{J}(\bar{u},B_R)+\mathcal{J}(\underline{u},B_R)+ 2\int_{B_{1/2}}\int_{B_{1/2}}\left(u(x-t\mbox{$\bf{v}$})-u(x)\right)_+\left(u(y-t\mbox{$\bf{v}$})-u(y)\right)_- \le \mathcal{J}(P_{t, \bf{v}}u,B_R)+\mathcal{J}(u,B_R).
\end{align}
Using this and the stability condition of $u$, and by adding $\mathcal{J}(P_{-t, \bf{v}}u,B_R)-3\mathcal{J}(u,B_R)$ to both sides of \eqref{6}, we have:
\begin{align*}
    \int_{B_{1/2}}\int_{B_{1/2}}\left(u(x-t\mbox{$\bf{v}$})-u(x)\right)_+\left(u(y-t\mbox{$\bf{v}$})-u(y)\right)_- dy dx \le & o(t^2)+\Delta^t_{\bf{v}\bf{v}}\mathcal{J}(u,B_R)\\
    \le & Ct^2\mathcal{J}^s(u,B_R)/R^2, \quad \mbox{by Lemma \ref{2.1},}
\end{align*}
Dividing $t^2$ on both sides and pass to limit as $t \rightarrow 0$, we can conclude \eqref{1}.

Define $A^{\pm}_{\bf{v}}:=\int_{B_{1/2}} \left(\partial _{\bf{v}}u(x)\right)_{\pm} dx$, then by \eqref{1} we have
\begin{align}
    \label{duanlian1}
 \min\{A^+_{\bf{v}}, A^-_{\bf{v}}\} \le \sqrt{A^+_{\bf{v}}A^-_{\bf{v}}} \le \sqrt{C\mathcal{J}^s(u,B_1)}.    
\end{align}
In addition, since $|u| \le 1$ and divergence theorem, 
\begin{align}
\label{duanlian2}
    |A^+_{\bf{v}}-A^-_{\bf{v}}|=\Big| \int_{B_{1/2}} \partial _{\bf{v}}u(x) \Big | \le C
\end{align}
Therefore, \eqref{duanlian1} and \eqref{duanlian2} yield
\begin{align*}
  \int_{B_{1/2}} |\partial_{\bf{v}}u(x)|dx =  A^+_{\bf{v}}+A^-_{\bf{v}}=|A^+_{\bf{v}}-A^-_{\bf{v}}|+2\min\{A^+_{\bf{v}}, A^-_{\bf{v}}\} \le C\left(1+\sqrt{\mathcal{J}^s(u,B_1)}\right). 
\end{align*}This proves \eqref{2}.
\end{proof}

\hfill\\

Now we are in a position to prove Proposition \ref{generalenergyestimates}.
\begin{proof}[proof of Proposition \ref{generalenergyestimates}]
For $0<s<1/2$, combining \eqref{lessthan1/2} for $R=1$ and Lemma \ref{2.2}, we have 
\begin{align}
\label{focus1}
    \int_{B_{1}}|\nabla u| \le C\left(1+\sqrt{C(1+\int_{B_4}|\nabla u|})\right)
\end{align}
By AM-GM inequality and Young's inequality, for $0<\delta<1$, whose choice depends on $d$ and $s$ which will be specified later on, there exists $C>0$ such that 
\begin{align}
\label{key}
    \int_{B_{1}}|\nabla u| \le \delta \int_{B_4}|\nabla u|+C/\delta.
\end{align}

Now we do the scaling argument. For any $x_0 \in \mathbb{R}^d$ and $\rho>0$ with $B_{\rho}(x_0) \subset B_{1}$, let $w(x):=u(x_0+\frac{\rho}{4} x)$, then $w$ is also a stable solution to \eqref{AC1} with $f(x)$ replaced by $\frac{\rho^{2s}}{4^{2s}} f(x_0+\frac{\rho}{4} x)$. Since the estimate above does not depend on $f$, by \eqref{key} we have
\begin{align*}
     \int_{B_{1}}|\nabla w| \le \delta \int_{B_4}|\nabla w|+C/\delta.
\end{align*}that is, 
\begin{align}
\label{iteration}
    \rho^{1-d} \int_{B_{\rho/4}(x_0)} |\nabla u| \le  \delta \rho^{1-d} \int_{B_{\rho}(x_0)}|\nabla u|+C/\delta. 
\end{align}
Then by Simon's Lemma proved in \cite{Simonlemma}, see also (see \cite[Lemma 3.1]{CSE} and \cite[Lemma 2.3]{FS17}), we can choose universal constant $\delta$ depending on $d$ and $s$ such that from \eqref{iteration}, we conclude that \begin{align}
\label{3}
    \int_{B_{1/2}} |\nabla u| \le C, 
\end{align}where $C$ depends only on $d$ and $s$.

Note that \eqref{3} is true for any stable solution $u$ to \eqref{AC1}, hence we can apply \eqref{3} for $u(x_0+2Rx)$, which is also a stable solution to \eqref{AC}, instead of $u(x)$, we have 
\begin{align}
\label{energyestimate1}
    \int_{B_R(x_0)}|\nabla u| \le CR^{d-1}, \quad \forall x_0 \in \mathbb{R}^d
\end{align}
By \eqref{3} and \eqref{lessthan1/2}, we have that for any stable solution $u$ to \eqref{AC1},
\begin{align}
\label{7}
    \mathcal{J}^s(u,B_{1/4}) \le C
\end{align}
Also by scaling property
\begin{align*}
    \mathcal{J}^s(u(x_0+4Rx), B_{1/4})=(4R)^{2s-d}\mathcal{J}^s(u,B_R(x_0)).
\end{align*}
Thus from \eqref{7} we conclude
\begin{align}
    \label{energyestimate2}
    \mathcal{J}^s(u,B_R(x_0)) \le CR^{d-2s}, \quad \forall x_0 \in \mathbb{R}^d
\end{align} These conclude \eqref{gle1} and \eqref{gle2} for the case $0<s<1/2$. 
\hfill\\

Next, we consider the case $1/2\le s<1$. 
By \eqref{biggerthan1/2} and \eqref{2} we have
\begin{align*}
     \int_{B_{1/2}} |\nabla u(x)| dx  \le C\left(1+\sqrt{C\left(1+L_0^{2s-1}\log (2L_0)\int_{B_2}|\nabla u|\right)}\right), 
\end{align*}where $L_0 \ge 2$ is an upper bound for $\Vert \nabla u \Vert_{L^{\infty}(B_1)}$.
Then similar to the argument \eqref{focus1}-\eqref{3}, we have
\begin{align}
\label{sorry}
    \int_{B_{1/2}} |\nabla u(x)| dx \le CL_0^{2s-1}\log(2L_0),
\end{align}
For any $x_0 \in \mathbb{R}^d$, since $u_R(x):=u(2Rx+x_0)$ is also a stable solution to \eqref{AC1} with $f$ replaced by $R^{2s}f$, by \eqref{sorry} we have
\begin{align}
\label{sorry'}
    \int_{B_{1/2}} |\nabla u_R(x)| dx \le CL_R^{2s-1}\log(2L_R),
\end{align}where $L_R \ge 2$ is an upper bound for $\Vert \nabla u_R \Vert_{L^{\infty}(B_1)}=2R\Vert \nabla u \Vert_{L^{\infty}(B_{2R}(x_0))}$. By \cite[Proposition 5.2]{CS17} and since $|u| \le 1$, $\Vert \nabla u \Vert_{L^{\infty}(\mathbb{R}^d)} \le C(d,s)M_0$, and thus we can choose $L_R \le CM_0R$. Hence by \eqref{sorry'} and scaling property we can conclude \eqref{gle1} for the case $1/2 \le s <1$.  Then by \eqref{biggerthan1/2}, and elliptic estimate $L_0 \le CM_0$, we derive \eqref{gle2} for the case $1/2\le s<1$. Note that the constant $C_2$ in \eqref{gle2} for the case $1/2<s<1$ does depend on $f$. However, when $s=1/2$, the constant in \eqref{gle2} does not depend on $f$.
\end{proof}

Now we are ready to validate Conjecture \ref{stable conjecture} for the case $d=2, \, 0<s<1$ in the following theorem.

\begin{theorem}
\label{stable2d}
If $u$ is a stable solution to \eqref{AC1} in $\mathbb{R}^2$, then $u$ is $1$-D.
\end{theorem}

\begin{proof}
By Proposition \ref{generalenergyestimates},  the RHS of \eqref{1} goes to zero as $R \rightarrow \infty$, and hence
\begin{align}
\label{8}
    \left(\int_{B_{1/2}} (\partial_{\bf{v}}u)_+(x)dx\right) \left( \int_{B_{1/2}} (\partial_{\bf{v}}u)_-(y)dy \right)=0. 
\end{align}
Then $u$ is monotone in $B_{1/2}$ along direction $\bf{v}$. Since \eqref{8} is true for any fixed direction $\bf{v}$ and any half ball, by the continuity of $u$ we conclude that $u$ is a $1$-D. 
\end{proof}

\section{Appendix: Proof of Remark \ref{connection}}

\begin{proof}[Proof of Remark \ref{connection}]
It is easy to see that  $u^{\pm \infty}$ are stable solutions in $\mathbb{R}^{d-1}$, and thus $1$-D solutions by hypothesis. By \cite[Theorem 2.12]{CY15},  $u^{\pm \infty}$ are monotone in some directions, and thus by \cite[Lemma 3.1]{DSV16}, $u^{\pm \infty}$ are local minimizers in $\mathbb{R}^{d-1}$. It is thus easy to see they are also local minimizers in $\mathbb{R}^d$.

We will show next that $u$ is also a local minimizer to $J$. We proceed as follows.

For any $\phi \in C_0^{\infty}(\Omega)$ where $\Omega$ is a bounded domain in $\mathbb{R}^d$, we consider local variation $\mathcal{J}(u+\phi)$. Let $m_1=\min\{u+\phi, u^{- \infty}\}$, $M_1=\max\{u+\phi, u^{- \infty}\}$. Hence outside $\Omega$, $m_1(x)=u^{- \infty}(x)$ and thus
\begin{align}
\label{chidao1}
    \mathcal{J}(m_1) \ge \mathcal{J}(u^{- \infty}).
\end{align}
By Lemma \ref{identity}, we have
\begin{align}
\label{chidao2}
    \mathcal{J}^s(u+\phi)+\mathcal{J}^s(u^{-\infty}) \ge \mathcal{J}^s(m_1)+\mathcal{J}^s(M_1).
\end{align}
It is easy to check 
\begin{align}
\label{chidao3}
    \mathcal{J}^P(u+\phi)+\mathcal{J}^P(u^{-\infty}) = \mathcal{J}^P(m_1)+\mathcal{J}^P(M_1).
\end{align}
Hence from \eqref{chidao2}-\eqref{chidao3} we have
\begin{align}
\label{chidao4}
    \mathcal{J}(u+\phi)+\mathcal{J}(u^{-\infty}) \ge \mathcal{J}(m_1)+\mathcal{J}(M_1)
\end{align}
By \eqref{chidao1} and \eqref{chidao4}, we have 
\begin{align}
\label{chidao5}
     \mathcal{J}(u+\phi)\ge \mathcal{J}(M_1).
\end{align}
Let $m_2=\min\{M_1, u^{\infty}\}$ and $M_2=\max\{M_1, u^{\infty}\}$. $m_2$ is a local variation of $u$ in the class 
\begin{align*}
   \{v : u^{-\infty} \le v \le u^{\infty}\}
\end{align*}
Simiarly as in the argument of \cite[Theorem 1]{PSV}, we can see that $u$ is a local minimizer in this class, and hence
\begin{align}
    \label{chidao6}
\mathcal{J}(m_2) \ge \mathcal{J}(u).
\end{align}
$M_2$ is a local variation of $u^{\infty}$ since outside $\Omega$, $M_2=u^{\infty}$. Hence
\begin{align}
    \label{chidao7}
\mathcal{J}(M_2) \ge \mathcal{J}(u^{\infty}).    
\end{align} By Lemma \ref{identity} we have
\begin{align}
\label{chidao8}
    \mathcal{J}(M_1)+\mathcal{J}(u^{\infty}) \ge \mathcal{J}(m_2)+\mathcal{J}(M_2),
\end{align} and hence by \eqref{chidao6}-\eqref{chidao7} it yields
\begin{align}
    \mathcal{J}(M_1) \ge \mathcal{J}(u).
\end{align}By \eqref{chidao5}, we obtain
\begin{align}
    \mathcal{J}(u+\phi) \ge \mathcal{J}(u).
\end{align}
Hence we have proved that $u$ is a minimizer as long as $u^{\pm \infty}$ are $1$-D stable solutions. Then that $u$ is an $1$-D solution when  $4 \le d \le 7, \, s \in (1/2-\epsilon_0,1)$ is from Theorem \ref{minimizer}. For $d \le 3$, this is because of Theorem \ref{longrange7}.
\end{proof}

$\mathbf{Acknowledgement}$  This research is partially supported by NSF grant DMS-1601885.

\end{document}